\documentclass[11pt]{article}

\usepackage[english]{babel}
\usepackage[applemac]{inputenc}
\usepackage{amssymb,amsmath,amsthm,latexsym}
\usepackage{tikz}

\usepackage[colorlinks=true,citecolor=blue]{hyperref}

\let \a = \alpha
\def \d {\displaystyle}
\let \t = \theta

\makeatletter
\newcommand*{\maison}[1]{%
  \mathord{%
    \mathpalette\@maison{#1}%
  }%
}
\newcommand*{\@maison}[2]{%
  \dimen@=\fontdimen8 %
      \ifx#1\scriptscriptstyle\scriptscriptfont
      \else\ifx#1\scriptstyle\scriptfont
      \else\textfont\fi\fi
      3 %
  \sbox0{%
    $#1%
      \vrule width\dimen@\relax
      \overline{%
        \kern2\dimen@
        \begingroup 
          #2%
        \endgroup
        \kern2\dimen@
      }%
      \vrule width\dimen@\relax
      \mathsurround=1.5\dimen@ 
    $%
  }%
  \ht0=\dimexpr\ht0-\dimen@\relax
  \dp0=\dimexpr\dp0+2\dimen@\relax
  \vbox{%
    \kern\dimen@ 
    \copy0 %
  }%
}

\newcommand{\bbC}{\mathbb{C}}

\newcommand{\bbQ}{\mathbb{Q}}
\newcommand{\bbR}{\mathbb{R}}
\newcommand{\bbZ}{\mathbb{Z}}

\DeclareMathOperator{\Real}{Re}
\DeclareMathOperator{\res}{res}
\DeclareMathOperator{\mHgt}{M}

\newtheorem{Th}{Theorem}

\newtheorem{Con}[Th]{Conjecture}
\newtheorem{Cor}[Th]{Corollary}

\newtheorem{Lem}[Th]{Lemma}

\theoremstyle{definition}
\newtheorem{Rem}{Remark}
\newtheorem*{Rem-nonum}{Remark}

\begin{document}

\title{Properties of Trinomials of Height at least $2$}

\author{V. Flammang and P. Voutier}

\date{}

\maketitle

\begin{abstract}
This paper is concerned with trinomials of the form $z^{n}+az^{m}+b$, where
$0<m<n$ are relatively prime integers and $a$ and $b$ are non-zero complex
numbers. Typically (but not exclusively), $a$ will be an integer with $|a| \geq 2$,
while $b=\pm 1$. Our main results cover the irreducibility, Mahler measure and
house of such trinomials.
\end{abstract}

\section{Introduction}

Trinomials have long been of interest to algebraists and number theorists.

For example, in 1907, O. Perron \cite{Pe} proved that the trinomial $z^{n}+az \pm 1$
is irreducible over $\bbQ$ for $|a| \geq 3$ and his results have been generalised
since then. Still, our understanding of the irreducibility of such polynomials is
weak. Theorem~\ref{thm:irreduc} below improves our knowledge when $b=\pm 1$
and we have substantial evidence that the truth is captured in Conjecture~\ref{conj:irreduc}
below.

The diophantine properties of trinomials are also of considerable importance.
C.J. Smyth proved that if $P$ is not a reciprocal polynomial, then the Mahler
measure of $P$ is bounded below by $\theta_{0}=1.324717\ldots$, the real zero
of the trinomial $z^{3}-z-1$. Furthermore, it is known that $\theta_{0}$ is also
the smallest Pisot number. The Mahler measure of $P$ is
also the smallest limit point of the Mahler measure of trinomials in $\bbZ[z]$
and one of the smallest known limit points for general polynomials in $\bbZ[z]$
\cite[Table~1, p. 412]{BM}.
See Subsections~\ref{subsec:Mahler} and \ref{subsec:house} for our results on
the Mahler measure of trinomials and, the related quantity, their house.

\subsection{Irreducibility of trinomials $z^{n}+az^{m} \pm 1 \in \bbZ[z]$}

We have the following result, generalising the result of Perron cited above.

\begin{Th}
\label{thm:irreduc}
If $a$, $m$ and $n$ are integers satisfying $n \geq 3$, $0<m<n$, $\gcd(m,n)=1$
and $|a| \geq n^{2}/3$, then $x^{n}+ax^{m}\pm 1$ is irreducible over $\bbQ$.
\end{Th}

In fact, from calculations we have performed, it appears that the following
much stronger conjecture is true. As in Perron's result, we appear to have
irreducibility once $|a|$ is greater than a small absolute constant regardless
of the degree of the middle term.

\begin{Con}
\label{conj:irreduc}
If $a$, $m$ and $n$ are integers satisfying $n \geq 3$, $0<m<n$, $\gcd(m,n)=1$
and $|a| \geq 5$, then $x^{n}+ax^{m}\pm 1$ is irreducible over $\bbQ$.

Furthermore, there are only finitely many reducible polynomials
of the form $x^{n}+ax^{m}\pm 1$ with $|a|=3$, $4$.
They are $x^{8} \pm 3x^{3}-1$, $x^{8} \pm 3x^{5}-1$,
$x^{13}+3x^{4}-1$, $x^{13}-3x^{4}+1$,
$x^{13}-3x^{6}-1$, $x^{13}+3x^{6}+1$,
$x^{13}+3x^{7} \pm 1$,
$x^{13}-3x^{9} \pm 1$,
$x^{14} \pm 4x^{5}-1$ and $x^{14} \pm 4x^{9}-1$.
\end{Con}

\begin{Rem}
Note that the condition that $\gcd(m,n)=1$ in both Theorem~\ref{thm:irreduc}
and Conjecture~\ref{conj:irreduc} is necessary. Examples from
Bremner \cite{Br} like $x^{33}+67x^{11}+1=\left( x^{3}+x+1 \right) \left(
x^{30}-\cdots -1 \right)$ in his Theorem on pages 153--154 and
$x^{6}+ \left( 4\mu^{4}-4\mu \right)x^{2}-1
= \left( x^{3}+2\mu x^{2}+2\mu^{2}x+1 \right)
\left( x^{3}-2\mu x^{2}+2\mu^{2}x-1 \right)$ for integers $\mu \neq 0,1$
in his Postscript on page~154 demonstrate this.
\end{Rem}


\subsection{Mahler Measure of trinomials $z^{n}+az^{m}+b \in \bbC[z]$}
\label{subsec:Mahler}

The {\it Mahler measure} of a polynomial $\d P(z)=a_{0}z^{n}+ \cdots +a_{n} = a_{0} \prod_{j=1}^{n} \left( z- \alpha_{j} \right) \in \bbC[z]$,
with $a_{0} \neq 0$, as defined by D. H. Lehmer \cite{L} in 1933,
is
\[
\mHgt(P)= \left| a_{0} \right| \prod_{j=1}^{n} \max \left( 1, \left| \alpha_{j} \right| \right).
\]

\begin{Rem-nonum}
Let $\alpha$ be a nonzero algebraic number. The {\it Mahler measure} of $\alpha$ is the Mahler measure of its minimal polynomial.
\end{Rem-nonum}

From a result of Kronecker, we know that, if $P$ is the minimal polynomial of an
algebraic integer then $M(P)=1$ if and only if $P$ is a cyclotomic polynomial.
D. H. Lehmer asked: does there exist a constant $c_{0} >0$ such that $M(P) > 1+ c_{0}$
for all $P$ not cyclotomic? The smallest known Mahler measure was found by
D.~H. Lehmer himself. It is the Mahler measure of the polynomial
$z^{10}+z^{9}+z^{7}-z^{6}-z^{5}-z^{4}-z^{3}+z+1$ for which the Mahler measure
is $1.176280\ldots$.

A polynomial $P$ is {\it reciprocal} if $z^{n}P( 1/z) = P(z)$ and an algebraic
number is {\it reciprocal} if its minimal polynomial is reciprocal. As stated
above, C.J. Smyth \cite {S1} solved Lehmer's problem when $P$ is not reciprocal.
In the general case, the
best known result is due to the second author \cite{V}, who proved that if $\alpha$ is
an algebraic integer of degree $n \geq 2$ and not a root of unity, then
\[
\mHgt(\alpha) \geq 1+ \frac{1}{4} \left( \frac{\log \log n}{\log n} \right)^{3}.
\]


In 2013, the first author \cite{F1} studied the Mahler measure of trinomials of
height $1$ and gave two criteria to identify those trinomials whose Mahler measure
is less than $1.381356\ldots=\mHgt \left( 1+z_{1}+z_{2} \right)$. In this same
work, she was able to prove a conjecture of Smyth on the Mahler measure of such
trinomials for $n$ sufficiently large compared to $m$. Stankov \cite{S}
was interested in trinomials of the type $z^{n}-az-1$ with $a \in (0,2]$ and he
presented the explicit expression by an integral of the limit of their Mahler
measure when $n$ tends to $\infty$. We generalise this result to any trinomial
of the form $z^{n}+az^{k}+b \in \bbC[z]$ in Theorem~\ref{thm:4}. In 2016,
J-L. Verger-Gaugry \cite{VG} studied
the family of trinomials $z^{n}+z-1$ for which he gave the asymptotic expansion
of the Mahler measure as a function of $n$ only. Here, we prove the following
results.


\begin{Th}
\label{thm:4}
Let $P(z)=z^{n}+az^{m}+b \in \bbC[z]$ with $ab \neq 0$.

\begin{enumerate}
\item
If $|a|-|b| \geq 1$, then $\d \lim_{n \rightarrow \infty} M(P) = |a|$.

\item
If $|b|-|a| \geq 1$, then $M(P) = |b|$.

\item
If $|a| + |b| \leq 1$, then $M(P)=1$.

\item
If $||a| -|b|| < 1 < |a|+ |b|$, then
\[
\lim_{n \rightarrow \infty} M(P) = \exp \left( \frac{1}{2 \pi} \int_{0}^{\gamma} \ln \left( |a|^{2}+2|ab| \cos (t) +|b|^{2} \right) dt \right),
\]
where we put $\gamma=\arccos \left( \dfrac{1-|a|^{2}-|b|^{2}}{2|ab|} \right)$.
\end{enumerate}
\end{Th}


As a result of the following theorem, we have exact expressions for $M(P)$ in
all cases in Theorem~\ref{thm:4} except (d).

\begin{Th}
\label{thm:VF-2}
Let $a$ and $b \neq 0$ be fixed complex numbers with $|a|-|b| \geq 1$.
For all integers $m$ and $n$ with $0<m<n$ and $\gcd(m,n)=1$,
we have
\begin{equation}
\label{eq:mahler-S}
\log M \left( z^{n}+az^{m}+b \right)
= \log |a| - \sum_{k \geq 1} \frac{1}{km} (-1)^{kn} \binom{kn-1}{km-1} \Real \left( b^{-km} (b/a)^{kn} \right).
\end{equation}
\end{Th}

This behaviour is quite different from what happens when $a, b=\pm 1$.
In 2007, Duke \cite{Duk} showed, for $0<m<n$ with $\gcd(m,n)=1$, that
\[
\log M \left( z^{n}+z^{m}+1 \right) = \log M(x+y+1)+\frac{c(n,m)}{n^{2}}
+ O \left( \frac{m}{n^{3}} \right),
\]
where $c(n,m)=-\pi \sqrt{3}/6$ if $3$ divides $m+n$ and $c(n,m)=\pi \sqrt{3}/18$
otherwise. The first author \cite{F1} proved that similar results hold for the other
trinomials with $\pm 1$ as their coefficients and used them as mentioned above.


\subsection{House of trinomials $z^{n}+az^{m} \pm 1 \in \bbZ[z]$}
\label{subsec:house}

Let $\alpha$ be a nonzero algebraic integer of degree $n$, with conjugates
$\a_{1}=\a,\ldots,\a_{n}$ and minimal polynomial, $P$. The {\it house} of $\a$
(and of $P$) is defined by:
\[
\maison{\a}=\d \max_{1 \leq i \leq n} \left| \a_{i} \right|.
\]

We have the inequality: $\maison{\a} \geq \mHgt(\a)^{1/n}$. In 1965, A. Schinzel
and H. Zassenhaus \cite{SZ} conjectured that there exists a constant $c>0$ such
that if $\a$ is not a root of unity then $\maison{\a} \geq 1 + c/n$. Thanks to
the polynomial $P(x)= x^{n}-2$, we see that $c \leq \log(2)$. In 1985, a result
of C.J. Smyth \cite{S1} led D. Boyd \cite{B1} to conjecture that $c$ should be
equal to $3/2 \log \t_{0}$ where $\t_{0}=1.324717\ldots$, and that this value
is attained too, via the polynomial $x^{3n}+x^{2n}-1$. The second author \cite{V}
proved that, if $\a$ is an algebraic integer of degree $n \geq 3$, not a root
of unity, then
\[
\maison{\alpha} \geq 1+\frac{1}{2n} ( \log \log n / \log n)^{3}.
\]

In 1991, E.M. Matveev \cite{M} proved that, if $\a$ is an algebraic integer of
degree $n \geq 2$, not a root of unity, then $\maison{\a} \geq \exp \left( \log \left( n +0.5 \right) /n^{2} \right)$.
Until Dimitov's very recent result (see below), the best-known asymptotic result
was given by A. Dubickas \cite{Dub}:
\[
\maison{\alpha} > 1+\frac{1}{n} \left( 64/\pi^{2}-\epsilon \right) \left( \log\log n/\log n \right)^{3}
\quad \text{for } n>n_{0}(\epsilon).
\] 

In 2007, G. Rhin and Q. Wu \cite{RW2} verified the conjecture of Schinzel and
Zassenhaus with the constant of Boyd up to degree $28$. They also established
that, if $\a$ is an algebraic integer of degree $n \geq 4$, not a root of unity,
then $\maison{\a} \geq \exp \left( 3 \log (n/3)/n^{2} \right)$ for $n \leq 12$,
and $\maison{\a} \geq \exp \left( 3 \log (n/2)/n^{2} \right)$ for $n \geq 13$.
It appears that the
result of \cite{RW2} improves Matveev for $n \geq 6$. Very recently (Dec. 2019),
Dimitrov \cite{Di} has proven the Schinzel-Zassenhaus conjecture with $c=\log(2)/4$.
His Theorem~1 states that when $P(z) \in \bbZ[z]$ is a non-cyclotomic monic polynomial
of degree $n>1$ that is irreducible over $\bbQ$ and $\alpha$ is a zero of $P(z)$,
then
\[
\maison{\alpha} \geq 2^{1/(4n)}=1+\frac{\log(2)}{4n} + O \left( 1/n^{2} \right).
\]

For trinomials, let $n \geq 2$ and $\theta_{n}$ be the unique real zero in $(0,1)$
of the trinomial $z^{n}+z-1$. By his method of asymptotic expansions of the zeros
mentioned above, J-L. Verger-Gaugry \cite{VG} obtained a direct proof of the conjecture
of Schinzel-Zassenhaus for $\theta_{n}^{-1}$, proving that
\[
\maison{\theta_{n}^{-1}} > 1 + \frac{ (\log n) \left( 1 - \frac{ \log \log n}{\log n} \right)}{n}.
\]

Here we focus on trinomials of the form $z^{n}+az^{m} \pm 1$,
where $0<m<n$ are relatively prime integers and $a$ is a positive real number.
Upon
replacing $z$ by $-z$, we find that all such polynomials are of the form\\
$R_{n,m,a}(z)=z^{n}-az^{m}+1$ with $a>0$, $m$ odd and $n$ even;\\
$S_{n,m,a}(z)=z^{n}+az^{m}-1$ with $a>0$ and $n$ odd; or\\
$T_{n,m,a}(z)=z^{n}-az^{m}-1$ with $a>0$.\\

\begin{Th}
\label{thm:house-R}
For any positive real number $a \geq 2$ and relatively prime positive integers
$0<m<n$ with $m$ odd and $n$ even, $r_{n,m,a}(z)$ has a real root $r_{n,m,a}^{(1)} \geq 1$
and
\[
r_{n,m,a}^{(1)} \geq 1 + \frac{\log(a-1)}{n-m}.
\]
Therefore, if we restrict $a$ to be an integer with $a \geq 2$, we have
\[
\maison{R_{n,m,a}} \geq 1 + \frac{\log(a-1)}{n-m}.
\]
\end{Th}

\begin{Th}
\label{thm:house-S}
For any positive real number $a \geq 2$ and relatively prime positive integers
$0<m<n$ with $m$ even and $n$ odd, $S_{n,m,a}(z)$ has a real root $s_{n,m,a}^{(3)} \leq -1$
and
\[
\left| s_{n,m,a}^{(3)} \right| \geq 1 + \frac{\log(a-1)}{n-m}.
\]

If we restrict $a$ to be an integer with $a \geq 2$, we have
\[
\maison{S_{n,m,a}} \geq 1 + \frac{\log(a-1)}{n-m}.
\]
\end{Th}

Unfortunately, when $m$ and $n$ are both odd, the single real zero of $S_{n,m,a}$
is between $0$ and $1$, while $\maison{S_{n,m,a}}>1$, so we are unable to obtain
a non-trivial lower bound for $\maison{S_{n,m,a}}$ in this case.

\begin{Th}
\label{thm:house-T}
For any positive real number $a \geq 2$ and relatively prime positive integers
$0<m<n$, $T_{n,m,a}(z)$ has a real root $t_{n,m,a}^{(1)}>1$ and
\[
t_{n,m,a}^{(1)} > 1 + \frac{\log(a)}{n-m}.
\]

Therefore, if $a$ is restricted to be an integer with $a \geq 2$, we have
\[
\maison{T_{n,m,a}} > 1 + \frac{\log(a)}{n-m}.
\]
\end{Th}
 
An algebraic integer $\alpha$ of degree $n$ is {\it{extremal}} if $\maison{\alpha}$
is the minimum of the houses of the algebraic integers of degree $n$. Denote by
$m(n)$ this minimum. From Smyth's example, $P_{n}(x)=x^{3n}+x^{2n}-1$, given
above (again, see \cite{B1,S1}), it is known that $m(n) \leq \theta_{0}^{3/(2n)}
= \maison{P_{n}}$. Here we have
 
\begin{Cor}
\label{cor:VF-2}
\begin{enumerate}
\item
For all positive integers $a$, $m$ and $n$ with $a \geq 2$, $\gcd(m,n)=1$, $m$
odd and $n$ even, if $R_{n,m,a}$ is irreducible over $\bbQ$, then its zeros
are not extremal.

\item
For all positive integers $a$, $m$ and $n$ with $a \geq 2$, $\gcd(m,n)=1$, $m$
even and $n$ odd, if $S_{n,m,a}$ is irreducible over $\bbQ$, then its zeros are
not extremal.

\item
For all positive integers $a$, $m$ and $n$ with $a \geq 2$ and $\gcd(m,n)=1$,
if $T_{n,m,a}$ is irreducible over $\bbQ$, then its zeros are not extremal.
\end{enumerate}
\end{Cor}

\section{Proof of Theorem~\ref{thm:irreduc}}

To prove Theorem~\ref{thm:irreduc}, we use the following result of Schinzel.

\begin{Lem}[Schinzel]
\label{lem:irreduc}
For positive integers $0<m<n$, put $m_{1}=m/\gcd(m,n)$ and $n_{1}=n/\gcd(m,n)$.

Let $a,b,c \in \bbZ \backslash \{0\}$, $\gcd(a,b,c)=1$. If $ax^{n}+bx^{m}+c$ is
reducible, then at least one of the following four conditions is satisfied:

\vspace*{1.0mm}

\noindent
{\rm (a)} $|b| \leq |a|^{m_{1}} |c|^{n_{1}-m_{1}}+1$;

\vspace*{1.0mm}

\noindent
{\rm (b)} $|b| \leq \dfrac{2m_{1}(n_{1}-m_{1})}{\log \left( 2m_{1} (n_{1}-m_{1}) \right)} |a|^{m/n} |c|^{(n-m)/n}$,
$\min \{ |a|, |c| \}=1$ and $\sqrt[p]{\max \{ |a|, |c| \}} \in \bbZ$ for some
prime $p|n_{1}$;

\vspace*{1.0mm}

\noindent
{\rm (c)} for some $q|\gcd(m,n)$, $q$ a prime or $q=$, $\sqrt[q]{|a|} \in \bbZ$,
$\sqrt[q]{|c|} \in \bbZ$ and if $q=2$, then $(-1)^{n_{1}}ac>0$, while if $q=4$,
then $ac>0$ and $n_{1} \equiv 0 \bmod 2$;

\vspace*{1.0mm}

\noindent
{\rm (d)} $4|\gcd(m,n)$, $ac>0$, $n_{1} \equiv 1 \bmod 2$ and either
$\sqrt[4]{|a|} \in \bbZ$, $\sqrt[4]{4|c|} \in \bbZ$ or
$\sqrt[4]{4|a|} \in \bbZ$, $\sqrt[4]{|c|} \in \bbZ$.
\end{Lem}

\begin{proof}
This is Theorem~9 on pages~12--13 of \cite{Sch}.
\end{proof}

\begin{proof}
To prove Theorem~\ref{thm:irreduc}, we suppose that $x^{n}+ax^{m}\pm 1$ is
reducible.

We apply Lemma~\ref{lem:irreduc} with $a=1$, $c=\pm 1$ and $b$ equal to our $a$ here.

Condition~(a) implies that $|a| \leq 2$.

Since $m_{1} \left( n_{1}-m_{1} \right) \leq n_{1}^{2}/4$ for all $0 \leq m_{1} \leq n_{1}$,
condition~(b) implies that $|a|<n^{2}/2$. A quick calculation for small relatively
prime values of $m$ and $n$ shows that we must have $|a|<0.321n^{2}$ for $n \geq 3$.

Since $\gcd(m,n)=1$, conditions~(c) and (d) cannot hold.
\end{proof}

\section{Proof of Theorem~\ref{thm:4}}

The proof requires two preliminary results.
 
\begin{Lem}
\label{lem:VF-1}
If $Q$ is a polynomial with complex coefficients, then
\[
\lim_{n \rightarrow \infty} \mHgt \left( z^{n} + Q(z) \right)
= \exp \left( \frac{1}{2 \pi} \int_{0}^{2 \pi} \log \max \left( 1, \left| Q \left( e^{it} \right) \right| \right) dt \right).
\]
\end{Lem}

\begin{proof}
D. Boyd \cite[Appendix~3]{B2} proved that $\d \lim_{n \rightarrow \infty} \mHgt \left( F \left(z, z^{n} \right) \right)
= \mHgt(F(z, w))$ if $F$ is a polynomial. Thus we have $\d \lim_{n \rightarrow \infty} \mHgt \left( z^{n} +Q(z) \right)
= \mHgt (w+Q(z))$. Now we apply Jensen's formula \cite{J} with respect to the
variable $w$
to obtain the lemma (also see \cite[equation~(21)]{B2}).
\end{proof} 

\begin{Cor}
\label{cor:VF-3}
\[
\d \lim_{n \rightarrow \infty} \mHgt \left( z^{n} + az^{m}+b \right)
= \exp\left( \frac{1}{2 \pi} \int_{0}^{2 \pi} \log \max \left( 1, \left| ae^{it}+b \right| \right) dt \right).
\]
\end{Cor}

\begin{proof}
Lemma~\ref{lem:VF-1} gives $\d \lim_{n \rightarrow \infty} \mHgt \left( z^{n} +az^{m} +b \right)
= \exp\left( \frac{1}{2 \pi} \int_{0}^{2 \pi} \log \max \left( 1, \left| ae^{itm}+b \right| \right) dt \right)$.
Furthermore,
\begin{align*}
\int_{0}^{2 \pi} \log \max \left( 1, \left| ae^{itm}+b \right| \right) dt
& = \frac{1}{m} \sum_{\ell=0}^{m-1} \int_{2 \pi \ell}^{2 \pi (\ell+1)} \log \max \left(1, \left| ae^{it}+b \right| \right) dt \\
&= \int_{0}^{2 \pi} \log \max \left( 1, \left| ae^{it}+b \right| \right) dt,
\end{align*}
proving the corollary.
\end{proof}

We are now able to prove Theorem~\ref{thm:4}. Let $P(z)=z^{n}+az^{m}+b$ be a
trinomial with $a,b \in \bbC$. We will need this preliminary calculation: if
$a =|a|e^{iu}$ and $b=|b|e^{iv}$ then 
$\left| ae^{it} +b \right|^{2} = \left| |a|e^{i(t+u-v)} + |b| \right|^{2}
=|a|^{2} + 2|a||b| \cos(t') + |b|^{2}$ where $t'=t+u-v$.

(a) $|a|-|b| \geq 1$:\\
We have $\left| ae^{it}+b \right|^{2} \geq (|a|-|b|)^{2} \geq 1$. Therefore,
$\max \left( 1, \left| ae^{it}+b \right| \right)=\left| ae^{it}+b \right|$, so
Corollary~\ref{cor:VF-3} gives
\[
\d \lim_{n \rightarrow \infty} \mHgt \left( z^{n} +az^{m} +b \right) = \mHgt (az+b)
= |a| \max (1, |b|/|a|)= \max(|a|,|b|).
\]

(b) $|b|-|a| \geq 1$:\\
We use Rouch\'{e}'s Theorem (again, see Corollary on page~153 \cite{Ah}). We want
to prove that all the zeros of $z^{n}P(1/z)=bz^{n}+az^{n-m}+1$
have absolute value at most $1$.

Suppose first that $|b|-|a|>1$. Putting $f(z)=bz^{n}$ and $g(z)=az^{n-m}+1$, we see
that $|g(z)| \leq |a|+1 < |b|=|f(z)|$ on the unit circle, since $|b|-|a|>1$. Hence
$P(z)=f(z)+g(z)$ and $f(z)$ have the same number of zeros inside the unit circle.
Since $f(z)$ has $n$ such zeros, so does $P(z)$. Hence all the zeros of $P$ have
absolute value at least $1$ and $\mHgt(P)=|b|$.

Now suppose that $|b|-|a|=1$ and let $\varepsilon>0$. We now consider $g(z)$ and
$f(z)$ on the circle of radius $1+\varepsilon$ centred at $0$. Now
$|g(z)| \leq |a|(1+\varepsilon)^{n-m}+1=|a|+|a|\varepsilon_{1}+1$ and
$|f(z)|=b(1+\varepsilon)^{n}=|b|+|b|\varepsilon_{2}=|a|+1+|a|\varepsilon_{2}+\varepsilon_{2}$,
for some $0<\varepsilon_{1}<\varepsilon_{2}$.
So $|f(z)|-|g(z)| \geq |a|\varepsilon_{2}+\varepsilon_{2}-|a|\varepsilon_{1}>0$.
Hence within any circle of radius $1+\varepsilon$ centred at $0$, $P(z)=f(z)+g(z)$
and $f(z)$ have the same number of zeros.
Since $f(z)$ has $n$ such zeros, so does $P(z)$. Taking the limit
as $\varepsilon \rightarrow 0$, we see that all the zeros of $P(z)$ have absolute
value at most $1$. Hence $\mHgt(P)=|b|$.

(c) $|a| + |b| \leq 1$:\\
We use Rouch\'{e}'s Theorem here too and proceed in the same way as above.
We put $f(z)=z^{n}+b$ and $g(z)=az^{m}$, and consider the cases $|a|+|b|<1$ and
$|a|+|b|=1$ separately.
%
%

(d) $|a-b|<1<|a+b|$:\\
We have $\left| ae^{it} +b \right| >1$ iff $t \in (0, \gamma)$ where
$\gamma = \d \arccos \left( \frac{1-|a|^{2}-|b|^{2}}{2|ab|} \right)$. From
Corollary~\ref{cor:VF-3}, we deduce
\[
\lim_{n \rightarrow \infty} \mHgt \left( z^{n} +az^{m} +b \right)
=\exp \left( \frac{1}{2 \pi} \int_{0}^{\gamma} \log \left( |a|^{2}+2|ab| \cos(t)+|b|^{2} \right) dt \right).
\]

\section{Proof of Theorem~\ref{thm:VF-2}}

We first need a lemma to show that we can interchange an integral and a sum that
arise in our proof.

\begin{Lem}
\label{lem:fubini}
Let $a$ and $b \neq 0$ be complex numbers with $|a|-|b| \geq 1$ and let
$m$ and $n$ be integers satisfying $0<m<n$. Then
\[
\int_{0}^{2\pi} \sum_{k \geq 1} \frac{1}{k} \left( \frac{e^{int}}{-ae^{imt}-b} \right)^{k} dt
=\sum_{k \geq 1} \int_{0}^{2\pi} \frac{1}{k} \left( \frac{e^{int}}{-ae^{imt}-b} \right)^{k} dt.
\]
\end{Lem}

\begin{proof}
According to Fubini's Theorem, this relationship holds if
\[
\int_{0}^{2\pi} \sum_{k \geq 1} \left| \frac{1}{k} \left( \frac{e^{int}}{-ae^{imt}-b} \right)^{k} \right| dt
<\infty
\, \text{ or } \,
\sum_{k \geq 1} \int_{0}^{2\pi} \left| \frac{1}{k} \left( \frac{e^{int}}{-ae^{imt}-b} \right)^{k} \right| dt
<\infty.
\]

We will use the first condition here:
\[
\int_{0}^{2\pi} \sum_{k \geq 1} \frac{1}{k} \left| \left( \frac{1}{ae^{imt}+b} \right)^{k} \right| dt.
\]

For $|a|-|b|>1$, then $\left| ae^{imt}+b \right|>1$, so
$\sum_{k \geq 1} \frac{1}{k} \left| \left( \frac{1}{ae^{imt}+b} \right)^{k} \right|$
converges and can be bounded from above independently of $t$. Since the integral
is over a bounded set, it is finite and the lemma holds.

For $|a|-|b|=1$, we use the fact that $\sum_{k \geq 1} z^{k}/k=-\log(1-z)$ if
$|z| \leq 1$ and $z \neq 1$. So we need to show that
\[
\int_{0}^{2\pi} \log \left( 1- \left| \frac{1}{ae^{imt}+b} \right| \right) dt < \infty.
\]

The only place where $\left| ae^{imt}+b \right|=1$ is where
$ae^{imt}=\left( |b|+1 \right) (-b)/|b|$. Note that there are $m$ such values of
$t$ satisfying $0 \leq t<2\pi$. Elsewhere $\left| ae^{imt}+b \right|>1$,
so away from these $m$ values of $t$, the integral is bounded. We only need to
show that in all neighbourhoods of each of these $m$ values of $t$, the integral
is also bounded. It suffices to consider only one such value of $t$, $t_{0}$,
and only neighbourhoods on one side of such a $t$. That is to show that
\[
\int_{t_{0}}^{t_{0}+\varepsilon} \log \left( 1- \left| \frac{1}{ae^{it}+b} \right| \right) dt < \infty,
\]
for $\varepsilon>0$.

Suppose that $a$ and $b$ are both real numbers.
Here $\left| ae^{it}+b \right|=\sqrt{a^{2}+b^{2}+2ab\cos(t)}$.

If $a$ and $b$ have different signs, then we need to consider what happens near $t_{0}=0$.
We want a lower bound for
$1- \dfrac{1}{\sqrt{a^{2}+b^{2}+2ab\cos(t)}}$ of the form $ct$ where $c>0$
for $t$ in some interval whose left endpoint is $t_{0}=0$.
%
There are
two cases to consider: (i) $a>0$, $b<0$ and $a=-b+1$; or (ii) $a<0$, $b>0$ and $a=-b-1$.
In the first case we consider the value of the function
$1- \dfrac{1}{\sqrt{a^{2}+b^{2}+2ab\cos(t)}}
=1- \dfrac{1}{\sqrt{2b^{2}-2b+1+2(1-b)b\cos(t)}}$ at $t=0$ (where it has the
value $0$) and $t=\pi/2$ (where it has the value $1- \dfrac{1}{\sqrt{2b^{2}-2b+1}}$).
This function is also convex for $t \in [0,\pi/2]$ (this follows from the fact
that $2b^{2}-2b+1+2(1-b)b\cos(t)$ is non-negative and increasing in this
interval, approaching $b^{2}+(b-1)^{2}$ as $t \rightarrow \pi/2$, and that
$1-1/\sqrt{x}$ is a convex function). So we have
\[
1- \frac{1}{\sqrt{a^{2}+b^{2}+2ab\cos(t)}}
\geq \frac{2}{\pi} \left( 1- \frac{1}{\sqrt{2b^{2}-2b+1}} \right) t,
\]
for $0 \leq t \leq \pi/2$. That is, we can take $c$ above to be
$\dfrac{2}{\pi} \left( 1- \dfrac{1}{\sqrt{2b^{2}-2b+1}} \right)$.

Now
$\int_{0}^{\varepsilon} \log(ct) dt =\varepsilon \log (c\varepsilon) - \varepsilon$,
so
\[
\lim_{\varepsilon \rightarrow 0} \int_{0}^{\varepsilon} \log(ct) dt
=\lim_{\varepsilon \rightarrow 0} \varepsilon \log (c\varepsilon) - \varepsilon
=\lim_{\varepsilon \rightarrow 0} \frac{\log (c\varepsilon)}{1/\varepsilon}
=\lim_{\varepsilon \rightarrow 0} \frac{1/\left( c\varepsilon \right)}{-1/\varepsilon^{2}}
=-\lim_{\varepsilon \rightarrow 0} \varepsilon/c=0,
\]
by L'H\^{o}pital's rule. Thus
\[
\int_{t_{0}}^{t_{0}+\varepsilon} \log \left( 1- \left| \frac{1}{ae^{it}+b} \right| \right) dt < \infty,
\]
and so, as argued above,
\[
\int_{0}^{2\pi} \log \left( 1- \left| \frac{1}{ae^{imt}+b} \right| \right) dt < \infty.
\]

The other cases, namely when $a<0$, $b>0$ and $a=-b-1$; and when
$a$ and $b$ have the same sign are treated in the very same way, proving the lemma
when $a, b \in \bbR$.

Furthermore, a more complicated version of this same argument can be used to
prove the lemma for any $a, b \in \bbC$ satisfying the conditions of the lemma.
\end{proof}

The idea of the proof of Theorem~\ref{thm:VF-2} is the same as the one used by
the first author in the study of trinomials of height $1$ (see \cite{F1}).

Using Jensen's formula, we put
\[
\lambda_{n,m,a}=\log M \left( z^{n}+az^{m}+b \right)
= \dfrac{1}{2\pi} \int_{0}^{2\pi} \log \left| e^{int}+ae^{imt}+b \right| dt.
\]

Note that $\lambda_{n,m,a}=\log M \left( -z^{n}-az^{m}-b \right)$ holds, so we
work with $-z^{n}-az^{m}-b$ here. This change simplifies slightly what follows.


Using Jensen's formula again and since $|a|>|b|$ (from our hypothesis that
$|a|-|b| \geq 1$), we have
\[
\dfrac{1}{2\pi}\int_{0}^{2\pi} \log \left| -ae^{imt}-b \right| dt
=\log M \left( -az^{m}-b \right)
=\log |a|.
\]

Thus
\begin{align*}
\lambda_{n,m,a}-\log |a|
&= \frac{1}{2\pi}
\int_{0}^{2\pi} \log \left| -e^{int}-ae^{imt}-b \right| - \log \left| -ae^{imt}-b \right| dt \\
&= \frac{1}{2\pi}
\int_{0}^{2\pi} \log \left| 1-\frac{e^{int}}{-ae^{imt}-b} \right| dt.
\end{align*}

Writing any $0 \neq x \in \bbC$ as $x=re^{i\theta}$ where $r>0$ and $|\theta| \leq \pi$,
we have $\log(x)=i\theta+\log(r)=i\theta+\log|r|$, so $\log |x| = \Real( \log(x))$
-- we have written this using the principal value of the logarithm function, but
it holds for any branch of the logarithm function.
Therefore,
\[
\lambda_{n,m,a}-\log |a|
= \frac{1}{2\pi}
\Real \int_{0}^{2\pi} \log \left( 1-\frac{e^{int}}{-ae^{imt}-b} \right) dt.
\]

Applying Lemma~\ref{lem:fubini} and the series expansion
\[
\log \left( 1-\frac{e^{int}}{-ae^{imt}-b} \right) \\
=- \sum_{k \geq 1} \frac{1}{k} \left( \frac{e^{int}}{-ae^{imt}-b} \right)^{k},
\]
we obtain
\begin{equation}
\label{eq:lambda-diff-S}
\lambda_{n,m,a}-\log |a|
= -\frac{1}{2\pi} \sum_{k \geq 1} \frac{1}{k}
\Real \underbrace{\int_{0}^{2\pi} e^{inkt} \left( -ae^{imt}-b \right)^{-k} dt}_{I_{k}}.
\end{equation}

From $z=e^{it}$, we have $dz=ie^{it}dt$, so
\[
I_{k} = -i\int_{|z|=1} z^{kn-1} \left( -az^{m}-b \right)^{-k} dz.
\]

We evaluate this integral using the Residue theorem:
\[
I_{k} = -i (2\pi i) \sum_{r: r^{m}=-b/a} \res \left( z^{kn-1} \left( -az^{m}-b \right)^{-k}; r \right),
\]
where the sum is over all $m$-th roots of $-b/a$ and $\res \left( f(z);z_{0} \right)$
is the residue of a function $f(z)$ at $z=z_{0}$.

Now observe that $z^{kn-1} \left( -az^{m}-b \right)^{-k}$ has one more pole when
we consider this function over the Riemann sphere. It has a pole at $\infty$ too.
Furthermore, the sum of the residues of a function over the Riemann sphere
equals $0$ \cite[equation~(2.7), p.~94]{Ca}, so
\begin{equation}
\label{eq:I-as-residue-S}
I_{k} = -2\pi \res \left( z^{kn-1} \left( -az^{m}-b \right)^{-k}; \infty \right),
\end{equation}

It is also known \cite[top of page~92]{Ca} that
\[
\res \left( f(z); \infty \right)
= -\res \left( \frac{1}{z^{2}}f \left( \frac{1}{z} \right); 0 \right),
\]
so we compute this quantity now. With 
$f(z)=z^{kn-1} \left( -az^{m}-b \right)^{-k}$, we have
\[
\frac{1}{z^{2}}f \left( \frac{1}{z} \right)
= \frac{z^{-1-kn+km}}{\left( -a-bz^{m} \right)^{k}}.
\]

The negative binomial series $\left( -a-bz^{m} \right)^{-k}$ equals
\[
\sum_{i=0}^{\infty} \binom{-k}{i} (-b)^{i}z^{im}(-a)^{-k-i}
=(-1)^{k} \sum_{i=0}^{\infty} \binom{-k}{i} b^{i}z^{im}a^{-k-i}.
\]

To get the the residue, we need the coefficient for $z^{kn-km}$ term in this sum.
I.e., $i=k(n-m)/m$. If $m \nmid kn$, then there is no such coefficient and the residue
is $0$. Otherwise, the coefficient of the $z^{kn-km}$ term is
\begin{align*}
\binom{-k}{k(n-m)/m} b^{k(n-m)/m} a^{-k-k(n-m)/m}
&=(-1)^{k(n-m)/m} \binom{k+k(n-m)/m-1}{k(n-m)/m} b^{k(n-m)/m} a^{-kn/m} \\
&=(-1)^{k(n-m)/m} \binom{kn/m-1}{kn/m-k} b^{k(n-m)/m} a^{-kn/m},
\end{align*}
since $\d \binom{-r}{s}=(-1)^{s} \binom{r+s-1}{s}$ for positive integers $r$ and
$s$. In fact, recall that $\gcd(m,n)=1$, so $m |kn$ if and only if $m|k$.
Therefore, from this expression and $\binom{kn/m-1}{kn/m-k}=\binom{kn/m-1}{k-1}$,
we obtain
\[
\res \left( z^{kn-1} \left( -az^{m}-b \right)^{-k}; \infty \right)
= \left\{
\begin{array}{ll}
-\displaystyle(-1)^{kn/m} \binom{kn/m-1}{k-1} b^{k(n-m)/m}a^{-kn/m} & \text{if $m|k$}, \\
0 & \text{otherwise.}
\end{array}
\right.
\]

Applying this expression to \eqref{eq:I-as-residue-S} and then \eqref{eq:lambda-diff-S},
we obtain
\[
\lambda_{n,m,a}-\log |a|
= -\sum_{k \geq 1} \frac{1}{km} (-1)^{kn} \binom{kn-1}{km-1} \Real \left( b^{-km} (b/a)^{kn} \right),
\]
which is equation~\eqref{eq:mahler-S}.

\section{Proofs of Theorems~\ref{thm:house-R}--\ref{thm:house-T} and Corollary~\ref{cor:VF-2}}

We start with a lemma about the location of the real zeros of the trinomials we
are considering.

\begin{Lem}
\label{lem:rt-location}
Let $a \geq 2$ be a real number, $m$ and $n$ are positive relatively prime integers
with $0<m<n$.

\begin{enumerate}
\item
Suppose also that $m$ odd and $n$ even.
$R_{n,m,a}(z)$ has two real zeros, $r_{n,m,a}^{(1)}$ and
$r_{n,m,a}^{(2)}$ satisfying
$r_{n,m,a}^{(1)}>1$ if $a>2$,
$r_{n,m,a}^{(1)}=1$ if $a=2$,
and $0<r_{n,m,a}^{(2)}<1$.

\item
Suppose also that $n$ is odd.
If $m$ is odd, then $S_{n,m,a}(z)$ has one real zero.
If $m$ is even. then $S_{n,m,a}(z)$ has three real zeros.
There is always one real zero, $0<s_{n,m.a}^{(1)}<1$.
If $m$ is even, then there is another real zero, $-1<s_{n,m,a}^{(2)}<0$ and a
third real zero, $s_{n,m,a}^{(3)}$.
If $a>2$, then $s_{n,m,a}^{(3)}<-1$.
If $a=2$, then $s_{n,m,a}^{(3)}=-1$.

\item
If $n$ is odd, then $T_{n,m,a}(z)$ has three real zeros if $m$ is odd and one
real zero if $m$ is even. There is always one zero, $t_{n,m.a}^{(1)}>1$. If $m$
is odd, then there is another real zero, $-1 \leq t_{n,m,a}^{(2)}<0$ and a third
real zero, $t_{n,m,a}^{(3)} \leq -1$.

If $n$ is even, then $T_{n,m,a}(z)$ has two real zeros. One, $t_{n,m,a}^{(1)}>1$
and the other, $-1 < t_{n,m,a}^{(2)}<0$.
\end{enumerate}
\end{Lem}

\begin{proof}
We prove only part~(c), as the proofs of parts~(a) and (b) are identical.

Using Descartes' Sign Rule, we find there is one positive real zero, since
there is one sign change among the coefficients of $T_{n,m,a}(z)$. Since
$T_{n,m,a}(1)=-a$ and $\lim_{z \rightarrow +\infty} T_{n,m,a}(z)=+\infty$, this
zero is strictly greater than $1$.

Applying Descartes' Sign Rule to $T_{n,m,a}(-z)=(-1)^{n}z^{n}+(-1)^{m-1}az^{m}-1$,
there are two negative zeros if $m$ and $n$ are both odd; no negative zeros if
$n$ is odd and $m$ is even; and one negative zero if $n$ is even (in which case,
$m$ is odd).

We have $T_{n,m,a}(-1)=(-1)^{n}-(-1)^{m}a-1$ equals $(-1)^{m-1}a=a$ if $n$ is even
(recalling that $m$ is odd in this case)
and $(-1)^{m-1}a-2$ if $n$ is odd. For $n$ even and since $T_{n,m,a}(0)=-1$, it
follows that the unique negative real zero is strictly between $-1$ and $0$.

Lastly, for $n$ odd and $m$ odd, we have $T_{n,m,a}(0)=-1$ and $T_{n,m,a}(-1)=a-2$.
So if $a>2$ and since $\lim_{z \rightarrow -\infty} T_{n,m,a}(z)=-\infty$, there
must be one in $(-1,0)$ and another in $(-\infty, -1)$. For $a=2$, there is a
zero at $z=-1$ and the second negative zero turns out to be less than $-1$ if
$m>n/2$ or in $(-1,0)$ if $m<n/2$.
\end{proof}

\subsection{Proof of Theorem~\ref{thm:house-R}}

When $n$ is even and $m$ is odd, from Lemma~\ref{lem:rt-location}(a),
$r_{n,m,a}^{(1)}$ is the unique real zero of the trinomial $R_{n,m,a}$ which
satisfies $r_{n,m,a}^{(1)} \geq 1$ and put $r_{n,m,a}^{(1)}=1+t$ with $t \geq 0$.
From the expression for $R_{n,m,a}$, we  we see that $(1+t)^{n}-a(1+t)^{m}+1=0$, so
$(1+t)^{n}-a(1+t)^{m}+1+(1+t)^{m} \geq 1$, i.e., $(1+t)^{n}-(a-1)(1+t)^{m} \geq 0$.
Thus
\[
(1 + t)^{n} \geq (a-1)(1 + t)^{m}.
\]

Let $t_{0}$ be the largest real number such that
\[
\left(1 + t_{0} \right)^{n} = (a-1) \left( 1 + t_{0} \right)^{m},
\]
then $t \geq t_{0}$ and $t_{0}=\exp \left( \log(a-1)/(n-m) \right)-1$. Using the
Taylor expansion of $\exp(x)$, $t_{0} \geq \log(a-1)/(n-m)$ and the theorem follows.

\subsection{Proof of Theorem~\ref{thm:house-S}}

When $n$ is odd and $m$ is even, recall from Lemma~\ref{lem:rt-location}(b)
that $s_{n,m,a}^{(3)}$ is the unique real zero of the trinomial $S_{n,m,a}$ which
satisfies $s_{n,m,a}^{(3)} \leq -1$ and put $s_{n,m,a}^{(3)}=-(1+t)$ with $t \geq 0$.
From the expression for $S_{n,m,a}$, we see that $(1+t)^{n}-a(1+t)^{m}+1=0$, so
as above
\[
(1 + t)^{n} \geq (a-1)(1 + t)^{m}
\]
and the proof follows as above too.

\subsection{Proof of Theorem~\ref{thm:house-T}}

From Lemma~\ref{lem:rt-location}(c), $t_{n,m,a}^{(1)}$ is the unique real
zero of $T_{n,m,a}$ which satisfies $t_{n,m,a}^{(1)}>1$ and put $t_{n,m,a}^{(1)}=1+t$
with $t>0$. From the expression for $T_{n,m,a}$, we know that
\[
(1 + t)^{n}>a(1 + t)^{m}.
\]

If
\[
\left(1 + t_{0} \right)^{n} = a \left( 1 + t_{0} \right)^{m},
\]
then $t>t_{0}$ and $t_{0}=\exp \left( \log(a)/(n-m) \right)-1$. Using the Taylor
expansion of $\exp(x)$, $t_{0}>\log(a)/(n-m)$ and our result follows.

\subsection{Proof of Corollary~\ref{cor:VF-2}}

We prove only part~(c) as the proof for the other two parts is identical.

It is clear that $m(n) \leq 2^{1/n}$.
We have $T_{n,m,a} \left( 2^{1/n} \right)= 1- a 2^{m/n} < 0$, since $a \geq 2$.
Since $T_{n,m,a}(z)>0$ for all $z>\theta_{m,n,a}$, we deduce $t_{n,m,a}^{(1)} > 2^{1/n}$,
which implies the non-extremality of $t_{n,m,a}^{(1)}$.

\noindent
V. Flammang:\\
UMR CNRS 7502. IECL, Universit\'e de Lorraine, site de Metz,\\
D\'epartement de Math\'ematiques, UFR MIM,\\
Ile du Saulcy, CS 50128. 57045 METZ cedex 01. FRANCE\\
E-mail address: valerie.flammang@univ-lorraine.fr

\vspace*{3.0mm}

\noindent
P. Voutier:\\
London, UK\\
E-mail address: Paul.Voutier@gmail.com
\end{document}